\documentclass[11pt,a4paper]{article}

\usepackage{inputenc}
\usepackage{amsmath}
\usepackage{bm}
\usepackage{bbold}
\usepackage{amsthm}

\usepackage{hyperref}

\title{A Recursive Equations Based Representation\\ for the $G/G/m$ Queue\thanks{Applied Mathematics Letters, 1994. Vol.~7, no.~3, pp.~73-78}
}

\author{Nikolai Krivulin\thanks{Faculty of Mathematics and Mechanics, St.~Petersburg State University, Bibliotechnaya Sq.2, Petrodvorets, St.~Petersburg, 198904 Russia}}

\date{}

\newtheorem{theorem}{Theorem}
\newtheorem{lemma}[theorem]{Lemma}

\begin{document}

\maketitle

\begin{abstract}
New recursive equations designed for the $  G/G/m  $ queue are presented.
These equations describe the queue in terms of recursions for the arrival and
departure times of customers, and involve only the operations of maximum,
minimum and addition.
\\

\end{abstract}

\section{Introduction}

As a representation of dynamics of queueing systems, recursive equations have
been introduced by Lindley in his classical investigation of the $  G/G/1  $
queue \cite{1}. In the last few years the representations based on recursive
equations have been extended to a variety of systems which consist of
single-server queues and operate according the first--come, first--served
(FCFS) discipline. Specifically, there are the equations designed for open and
close tandem queues with both infinite and finite buffers \cite{2,3,4}.
Recursive equations have been also derived to represent more complicated
systems of $  G/G/1  $ queues, including acyclic fork-join networks
\cite{5,6} and closed networks with a general deterministic routing mechanism
\cite{7,8,9}.

Recursive equations find a wide application in recent works on both analytical
study and simulation of queueing systems. As an analytical tool, they were
exploited to investigate system performance measures \cite{2,7,10} and
estimates of their gradients \cite{7,8,9,11}. In \cite{5,6} recursive equations
based representations have provided the means for establishing the stability
conditions and deriving bounds on system performance in a class of queueing
systems. Finally, these representations made it possible to develop efficient
algorithms of queueing systems simulation \cite{3,4} as well as powerful
methods of estimating gradients of system performance measures \cite{7,9,11}.

In the queueing theory, the $  G/G/m  $ queue is normally represented by
using the recursions introduced by Kiefer and Wolfowitz in \cite{12}.
Expressed in rather general terms, these recursions may be inconvenient to use
if an explicit form of representation is required. The purpose of this paper
is to present new recursive equations designed for the $  G/G/m  $ queue.
These equations describe the queue in terms of the arrival and departure times
of customers, and involve only the operations of maximum, minimum and
addition.

The rest of the paper is organized as follows. In Section~2 we briefly
describe both the $  G/G/1  $ and $  G/G/m  $ queues, and give preliminary
analysis of their representations. Section~3 includes two technical lemmas
which offer useful representations for the ordering operator. Finally, in
Section~4 we present our main result providing an explicit recursive
representation for the $  G/G/m  $ queue.

\section{Definitions and Preliminary Analysis}

Recursive equations as a representation of queueing systems were originally
expressed in terms of waiting times \cite{1,12}. Equations of this classical
type continue in use (see, e.g., \cite{10}); however in many recent works
\cite{2,3,4,5,6,7,8,9,11} one can find the equations describing system dynamics through
recursions for arrival and departure times. The last approach to the
representation of queues may be considered as a useful generalization of the
classical one, and it will be applied below to derive recursive equations for
the $  G/G/m  $ queue.

We start with the equations for the $  G/G/1  $ queue, which provide the
basis for representing more complicated systems including the $  G/G/m  $
queue. To set up these equations, consider a single server queue with infinite
buffer capacity and the FCFS queue discipline. Denote the interarrival time
between the $k$th customer and his predecessor by $  \alpha_{k}  $, and the
service time of the $k$th customer by $  \tau_{k} $. We assume
$  \alpha_{k} \geq 0  $ and $  \tau_{k} > 0 $, for any
$  k = 1, 2, \ldots $ Furthermore, let $  A_{k}  $ be the $k$th arrival
epoch to the queue, and $  D_{k}  $ be the $k$th departure epoch from the
queue. As is customary, the sequences $  \{\alpha_{k}\}_{k \geq 1}  $ and
$  \{\tau_{k}\}_{k \geq 1}  $ are assumed to be given, whereas
$  \{A_{k}\}_{k \geq 1}  $ and $  \{D_{k}\}_{k \geq 1}  $ are considered
unknown. Finally, provided that the queue starts operating at time zero, it is
convenient to set $  A_{k} \equiv 0  $ and $  D_{k} \equiv 0  $ for all
$  k \leq 0 $.

Using these notations, the recursive equations for the $  G/G/1  $ queue may
be written as \cite{2,3,4}
\begin{equation}
\begin{split}
A_{k} & = A_{k-1} + \alpha_{k} \\
D_{k} & = (A_{k} \vee D_{k-1}) + \tau_{k},
\end{split}
\label{tag1}
\end{equation}
where $  \vee  $ denotes the maximum operator, $  k=1,2, \ldots $

The first equation in (\ref{tag1}) is trivial. To understand the second
recursion it is sufficient to see that the term $  A_{k} \vee D_{k-1}  $
determines the service initiation time for the $k$th customer. Clearly, the
service of this customer may be initiated not earlier than he arrives at the
server. If the $k$th customer finds the server busy, he has to wait until
service of the $(k-1)$st customer completes. In other words, the time of the
$k$th initiation of service coincides with the maximum out of $  A_{k}  $
and $  D_{k-1}  $.

Taking (\ref{tag1}) as the starting point, we now turn to the discussion of
multi--server queues. Consider a queueing system with an infinite buffer and
$  m $, $  1 \leq m < \infty $, servers operating in parallel. When a
customer arrives, he occupies any one of those servers which are not busy. If
there are no free servers at his arrival, the customer takes his place in the
buffer and starts waiting to be served. His service begins as soon as any one
of the servers is unoccupied, provided that all his predecessors have left
the buffer.

We may extend all the definitions introduced above to this queueing system.
Note, however, that in the multi--server queue the $k$th departure time may
not coincide with the completion time of the $k$th customer. One can consider
the first customer as an example. Being the first to initiate service, he may
prove to be the $k$th to depart from the system, for any $  k \geq 1 $. To
take account of the distinction between these times, we further introduce the
notation $  C_{k}  $ for the completion time of the $k$th customer,
$  k=1,2, \ldots $

Since upon their service completions the customers immediately leave the
system, it is easy to see that both the sequences of completion times and
departure times are constituted by the same elements. In fact, the sequence
$  \{D_{k}\}_{k \geq 1}  $ is simply the sequence
$  \{C_{k}\}_{k \geq 1}  $ arranged in ascending order. Therefore, an
operator which produces ordered values is required to express the relation
between these sequences.

The idea to apply some ordering operator in recursive equations based
representations of the $  G/G/m  $ queue had its origin in \cite{12}.
However, this operator, as it has been introduced in \cite{12} and occurs in
recent works (see, e.g., \cite{6}), is expressed in general terms rather than
in an explicit form. It will be shown in the next section how the ordering
operator may be represented in closed form as a function of values being
ordered.

\section{Representations of the Ordering Operator}

Let $  \{r_{j}\}_{j = 1}^{n} = \{ r_{1}, \ldots, r_{n} \}  $ be a finite
set (sequence) of real numbers. Suppose that we arrange its elements in order
of increase, and denote the $k$th smallest element by $  r_{(k)}^{n} $. If
there are elements of an equal value, we count them repeatedly in an arbitrary
order. Finally, we introduce the notation $  \wedge  $ for the minimum
operator, and set $  r_{(k)}^{n} = -\infty $ for all $  k \leq 0 $.

\begin{lemma}
For each $  k = 1, \ldots, n  $, the value of
$  r_{(k)}^{n}  $ is given by
\begin{equation}
r_{(k)}^{n} = \bigwedge_{1 \leq j_{1} < \cdots < j_{k} \leq n}
(r_{j_{1}} \vee \cdots \vee r_{j_{k}}).
\label{tag2}
\end{equation}
\end{lemma}

The proof of this statement can be found in \cite{9}. Now suppose that a new 
element $  r_{n+1}  $ is added to $  \{ r_{j} \}_{j=1}^{n} $. For the 
expanded set $  \{r_{j}\}_{j = 1}^{n+1} $, we denote the $k$th smallest 
element by $  r_{(k)}^{n+1} $. The next lemma is intended to relate the 
ordered values of the set $  \{ r_{j} \}_{j=1}^{n+1} $ to those of
$  \{ r_{j} \}_{j=1}^{n} $.

\begin{lemma}
For each $  k = 1, \ldots, n $, it holds
\begin{equation}
r_{(k)}^{n+1} = r_{(k)}^{n} \wedge (r_{(k-1)}^{n} \vee r_{n+1}).
\label{tag3}
\end{equation}
\end{lemma}
\begin{proof}
To derive (\ref{tag3}), we first apply Lemma~1 to represent
$  r_{(k)}^{n+1} $ in the form
$$
r_{(k)}^{n+1} = \bigwedge_{1 \leq j_{1} < \cdots < j_{k} \leq n+1}
(r_{j_{1}} \vee \cdots \vee r_{j_{k}}).
$$
The terms being minimized may be rearranged to write this relation as
$$
r_{(k)}^{n+1} = \left( \bigwedge_{1 \leq j_{1} < \cdots < j_{k} \leq n}
                (r_{j_{1}} \vee \cdots \vee r_{j_{k}}) \right) \wedge
                \left( \bigwedge_{1 \leq j_{1} < \cdots < j_{k} = n+1}
(r_{j_{1}} \vee \cdots \vee r_{j_{k}}) \right).
$$
Clearly, the first term on the right-hand side of the above relation just
represents $  r_{(k)}^{n} $. It remains to factor out the common member
$  r_{j_{k}} = r_{n+1}  $ from the second term, and to apply Lemma~1 once
again so as to get (\ref{tag3}):
\begin{multline*}
r_{(k)}^{n+1} = r_{(k)}^{n} \wedge
                \left( \bigwedge_{1 \leq j_{1} < \cdots < j_{k-1} \leq n}
                (r_{j_{1}} \vee \cdots \vee r_{j_{k-1}}) \vee r_{n+1} \right)
\\                
              = r_{(k)}^{n} \wedge (r_{(k-1)}^{n} \vee r_{n+1}).
\qedhere              
\end{multline*}
\end{proof}

To conclude this section, note that the above representations of the ordering
operator involve only the operations of maximum and minimum, and appear to be
particularly suited for use in the recursive equations under discussion, which
are actually expressed in terms of similar operations.

\section{Recursive Equations for the $  G/G/m  $ Queue}

We are now in a position to present the main result of the paper. We start the
derivation of the equations representing the $  G/G/m  $ queue with the
observation that the arrival process in this queue system is no different from
that in the $  G/G/1 $. Therefore, the first equation at (\ref{tag1})
$$
A_{k} = A_{k-1} + \alpha_{k}
$$
remains unchanged.

Obviously, to calculate the completion time of the customer which is the $k$th
to arrive to the system, one has to add his service time $  \tau_{k}  $
to the time of his service initiation. Similarly as in the above analysis of
the $  G/G/1  $ queue, let us examine the possibilities for this customer to
initiate his service. Firstly, the customer may find one or more servers free
at his arrival. In this case he receives service immediately at the time
$  A_{k} $.

Suppose now that at the arrival of the $k$th customer all the servers are
found to be busy. If there are no other customers waiting for service, he
occupies that server which becomes free earlier. Clearly, the customers being
served unoccupy servers according to the sequence of the departure times
$  D_{k-m}, \ldots, D_{k-1} $. We may therefore conclude that, as the earliest
time in this sequence, $  D_{k-m} $ just represents the service initiation
time of the $k$th customer. It is easy to see that this conclusion remains the
same in case there are previous customers still waiting for service, when the
$k$th customer arrives. Finally, for the completion time of this customer, we
have
$$
C_{k} = (A_{k} \vee D_{k-m}) + \tau_{k}.
$$

As an important consequence of the above equation, one can state that
$  C_{k} > D_{k-m} $. By renaming the indices, this relation may be
rewritten as
\begin{equation}
D_{k} < C_{k+m}.
\label{tag4}
\end{equation}

To represent the $  G/G/m  $ queue completely, we have to define the
departure time $  D_{k}  $ by a suitable equation. One can conclude from the
discussion in Section~2 that $  D_{k}  $ coincides with the $k$th smallest
element of the set $  \{C_{j}\}_{j \geq 1} $. Moreover, as it follows from
(\ref{tag4}), it is sufficient to examine only the finite subset
$  \{C_{1}, \ldots, C_{k+m-1} \}  $ so as to determine $  D_{k} $. Let
$  C_{(k)}^{k+m-1}  $ be the $k$th smallest element of this set. We may now
define $  D_{k} = C_{(k)}^{k+m-1} $.

By applying (\ref{tag3}) and (\ref{tag4}), and taking into account that
$  D_{k-1} = C_{(k-1)}^{k+m-2} $, we successively get
\begin{multline*}
C_{(k)}^{k+m-1}
      = C_{(k)}^{k+m-2} \wedge ( C_{(k-1)}^{k+m-2} \vee C_{k+m-1} )
\\      
      = C_{(k)}^{k+m-2} \wedge ( D_{k-1} \vee C_{k+m-1} )
      = C_{(k)}^{k+m-2} \wedge C_{k+m-1}.
\end{multline*}
Finally, using (\ref{tag2}), we arrive at
$$
D_{k} = C_{(k)}^{k+m-1}
      = \bigwedge_{ 1 \leq j_{1} < \cdots < j_{k} \leq k+m-2 }
        ( C_{j_{1}} \vee \ldots \vee C_{j_{k}} ) \wedge C_{k+m-1}.
$$
We summarize the above results as follows.
\begin{theorem}
With the notations introduced in Section~2, the dynamics
of the $\ G/G/m  $ queue is defined by the recursive equations
\begin{equation}
\begin{split}
A_{k} & = A_{k-1} + \alpha_{k} \\
C_{k} & = (A_{k} \vee D_{k-m}) + \tau_{k} \\
D_{k} & = \bigwedge_{ 1 \leq j_{1} < \cdots < j_{k} \leq k+m-2 }
( C_{j_{1}} \vee \ldots \vee C_{j_{k}} ) \wedge C_{k+m-1},
\end{split}
\label{tag5}
\end{equation}
for each $  k = 1,2, \ldots $
\end{theorem}

In conclusion, let us consider two consequences which follow immediately from
the theorem. Clearly, if $  m=1  $, it results from the last equation at
(\ref{tag5}) that $  C_{k} \equiv D_{k}, \; k = 1,2, \ldots $ In this case, the
above equations are reduced to the recursions (\ref{tag1}) representing the
$  G/G/1  $ queue.

As another consequence of the obtained representation, one may produce the
equations describing the $  G/G/2  $ queue. Putting $  m=2  $ in
(\ref{tag5}), it is not difficult to arrive at
\begin{align*}
A_{k} & = A_{k-1} + \alpha_{k} \\
C_{k} & = (A_{k} \vee D_{k-2}) + \tau_{k} \\
D_{k} & = (C_{1} \vee \cdots \vee C_{k}) \wedge C_{k+1}.
\end{align*}

\bibliographystyle{utphys}

\bibliography{A_recursive_equations_based_representation_for_the_GGm_queue}

\end{document}